\documentclass[12pt]{amsart}
\usepackage{amsmath,amsthm,amssymb}
\usepackage{tikz}
\usepackage{hyperref} 
\usetikzlibrary{arrows.meta}

\usepackage{a4wide}
\usepackage{enumerate}
\usepackage{multicol, multirow, tasks}

\newtheorem{theorem}{Theorem}
\newtheorem{proposition}[theorem]{Proposition}
\newtheorem{lemma}[theorem]{Lemma}

\theoremstyle{definition}

\newcommand{\g}[1]{\mathfrak{#1}}

\theoremstyle{remark}
\newtheorem{remark}[theorem]{Remark}

\DeclareMathOperator{\tr}{tr}

\DeclareMathOperator{\ad}{ad}

\begin{document}
\title[Minimal homogeneous submanifolds of $\mathbb{C} H^n$]{Minimal homogeneous submanifolds of \\ complex hyperbolic~spaces}
\author[A.~Cidre-D\'iaz]{\'Angel~Cidre-D\'iaz}
\address{CITMAga, 15782 Santiago de Compostela, Spain.\newline\indent Department of Mathematics, Universidade de Santiago de Compostela, Spain}
\email{angel.cidre.diaz@usc.es}
\author[M.~Dom\'inguez-V\'azquez]{Miguel~Dom\'inguez-V\'azquez}
\address{CITMAga, 15782 Santiago de Compostela, Spain.\newline\indent Department of Mathematics, Universidade de Santiago de Compostela, Spain}
\email{miguel.dominguez@usc.es}

\begin{abstract}
We classify minimal extrinsically homogeneous submanifolds of complex hyperbolic spaces. 
\end{abstract}

\thanks{The authors have been supported by grant PID2022-138988NB-I00 funded by MICIU/AEI/10.13039/501100011033 and by ERDF, EU, and by project ED431C 2023/31 (Xunta de Galicia, Spain). The first author acknowledges support of an FPU fellowship. 
} 
\maketitle
\vspace{-1cm}

\section{Introduction}
A connected submanifold of a Riemannian manifold is called (extrinsically) homogeneous if it is an orbit of a Lie group isometric action, and minimal if its mean curvature vanishes.  

In this paper, we give the first classification of homogeneous minimal submanifolds in a complete family of symmetric spaces of nonconstant curvature, namely in complex hyperbolic spaces $\mathbb{C} H^n$. The following result is a consequence of this classification.
\begin{theorem}\label{th:short}
	Let $S$ be a homogeneous submanifold of a complex hyperbolic space $\mathbb{C} H^n$. Then $S$ is minimal if and only if it is totally geodesic or a minimal leaf of an isoparametric family of hypersurfaces of $\mathbb{C} H^n$.
\end{theorem}
An isoparametric family of hypersurfaces is a decomposition of a Riemannian manifold into equidistant hypersurfaces of constant mean curvature and possibly one or two focal (a fortiori, minimal) submanifolds of codimension greater than one. Isoparametric families of hypersurfaces in $\mathbb{C} H^n$ are classified~\cite{DRDVSL:adv}. Except for the horosphere foliation of $\mathbb{C} H^n$, each isoparametric family of $\mathbb{C} H^n$ has exactly one minimal leaf, which is either focal or the so-called Lohnherr hypersurface (see Remark~\ref{rem:isop}). Such minimal leaves are homogeneous, even though most isoparametric families of $\mathbb{C} H^n$ are inhomogeneous.

The investigation of minimal homogeneous submanifolds dates back at least to the works of Hsiang \cite{Hsiang} and Takahashi~\cite{Takahashi}. In particular, every compact homogeneous space can be embedded as a minimal, extrinsically homogeneous submanifold of a round sphere $\mathbb{S}^n$ of sufficiently high dimension, showing that such submanifolds are abundant in spheres. Moreover, by results of Hsiang and Lawson~\cite{HsiangLawson}, compact homogeneous spaces also admit large families of minimal homogeneous submanifolds. However, no complete classification is known in this setting, even for spheres.

In nonpositive constant curvature, the situation is considerably more rigid: minimal homogeneous submanifolds of Euclidean and real hyperbolic spaces are precisely the totally geodesic ones, as shown by Di Scala and Olmos~\cite{DiScala}, \cite{discala-olmos}. Beyond the constant curvature setting, however, symmetric spaces of noncompact type admit a wide variety of minimal homogeneous submanifolds that are not totally geodesic (see for example \cite{DRDV:advances}, \cite{BS:cag}, \cite{Tamaru}). General results relating totally geodesic and minimal submanifolds in this context were obtained by Alekseevsky and Di Scala~\cite{ADS}. Nevertheless, despite this growing body of examples and partial results, a complete classification in the nonconstant curvature case remains out of reach.

The main goal of this paper is to initiate a systematic approach to this problem by establishing a complete classification in the first natural and fundamental case, namely that of complex hyperbolic spaces.
The resulting examples are described in terms of the root space and Iwasawa decompositions of the Lie algebra of the isometry group of $\mathbb{C} H^n$. To state the classification precisely, we briefly recall some known facts and fix notation; further details can be found in \S~\ref{section:preliminaries}.

It is well known that the connected component of the identity of the isometry group of $\mathbb{C} H^n$ is a finite quotient of $G=\mathrm{SU}(1,n)$. Let $\g{g}=\g{k}\oplus\g{a}\oplus\g{n}$   be the Iwasawa decomposition of the Lie algebra of $G$. Here $\g{k}$ is the Lie algebra of the isotropy group $K=\mathrm{S(U(1)\times U}(n))$ at some fixed base point $o\in \mathbb{C} H^n$, $\g{a}$ is a $1$-dimensional subspace of the orthogonal complement of $\g{k}$ in $\g{g}$ (with respect to the Killing form), and $\g{n}$ is a nilpotent subalgebra obtained as the sum of the positive root spaces of $\g{g}$ with respect to $\g{a}$. In this case, $\g{n}=\g{g}_\alpha\oplus\g{g}_{2\alpha}$, where $\alpha$ is a chosen simple root. It is also known that the metric and complex structure of $\mathbb{C} H^n$ make $\g{g}_\alpha$ isomorphic to the complex Euclidean space~$\mathbb{C}^{n-1}$. Below, whenever we refer to a  subspace $\g{m}$ of $\g{g}_\alpha \cong \mathbb{C}^{n-1}$, we mean a real vector subspace of $\g{g}_\alpha$, where $\g{g}_\alpha$ is regarded as a $(2n-2)$-dimensional real vector space.

\begin{theorem}\label{Thm1}
    A homogeneous submanifold of $\mathbb{C}H^n$ of positive dimension is minimal if and only if it is congruent to the orbit $H \cdot o$ through the base point $o \in \mathbb{C}H^n$ of a connected subgroup $H$ of $G$ with one of the following Lie algebras:
    \begin{enumerate}
        \renewcommand{\labelenumi}{(\alph{enumi})}
        \item $\g{h} = \g{a} \oplus \g{m}$, where $\g{m}$ is a totally real subspace of $\g{g}_\alpha \cong \mathbb{C}^{n-1}$.
        \item $\g{h} = \g{a} \oplus \g{m} \oplus \g{g}_{2\alpha}$, where $\g{m}$ is any subspace of $\g{g}_\alpha \cong \mathbb{C}^{n-1}$.
    \end{enumerate}
\end{theorem}

This classification extends the partial characterization of minimal submanifolds of $\mathbb{C}H^n$ arising as orbits of subgroups of the solvable Iwasawa group $AN$ with Lie algebra $\g{a}\oplus\g{n}$,  obtained by Alekseevsky and Di Scala~\cite[Lemma 7.1]{ADS}. A main challenge in the general setting stems from the need to deal with diagonal subgroups of parabolic subgroups 
$K_0AN$ of $G=\mathrm{SU}(1,n)$, where the presence of a nontrivial compact component $K_0$ introduces considerably greater complexity. Accordingly, the bulk of the technical work in this paper is devoted to the study of such diagonal subgroups, leading to new tools that may be useful for the further study of homogeneous submanifolds in rank-one symmetric spaces.

Apart from the explicit classification  in Theorem~\ref{Thm1}, we also provide in Theorem~\ref{th:short} a characterization of the examples in terms of isoparametric hypersurfaces in $\mathbb{C} H^n$. This characterization is a consequence of the following result. 

\begin{theorem}\label{th:isop}
    A homogeneous submanifold $S\subsetneq\mathbb{C}H^n$  is minimal if and only if it is
    \begin{enumerate}
        \item a totally geodesic real hyperbolic subspace $\mathbb{R} H^k$, $k\in\{1,\dots, n-1\}$, or
        \item a Lohnherr hypersurface, or
        \item the focal submanifold of an isoparametric family of hypersurfaces of $\mathbb{C} H^n$.
    \end{enumerate}
\end{theorem}
Let us observe that totally geodesic submanifolds of $\mathbb{C} H^n$ are either real or complex hyperbolic subspaces. It is well known that tubes around a totally geodesic complex hyperbolic subspace $\mathbb{C} H^k$, $k\in\{0,\dots, n-1\}$, or around a totally geodesic real hyperbolic subspace $\mathbb{R} H^n$ of maximal dimension, define a homogeneous isoparametric family. Thus, such totally geodesic submanifolds $\mathbb{C} H^k$ and $\mathbb{R} H^n$ are included in item (3) of Theorem~\ref{th:isop}. We also note that a Lohnherr hypersurface \cite{LR} is known to be the only minimal homogeneous hypersurface of $\mathbb{C} H^n$ up to congruence (as follows for example from~\cite{BDR:gd}). For an explicit description of the Lohnherr hypersurface and the possible focal submanifolds of isoparametric families (as classified in~\cite{DRDVSL:adv}), we refer the reader to Remark~\ref{rem:isop}. 

The paper is organized as follows. In Section~\ref{section:preliminaries}, we recall some basic facts about $\mathbb{C}H^n$, derive an explicit formula for the second fundamental form of the orbits of subgroups of a parabolic group, and prove the reduction of the classification problem to such subgroups. Then Section~\ref{sec:classification} is devoted to the proof of the classification result stated in Theorem~\ref{Thm1}, from where we will also deduce Theorems~\ref{th:short} and~\ref{th:isop}.

\section{Preliminary results}\label{section:preliminaries}

In this section we fix notation, collect basic facts, and review or derive certain results  used throughout the paper. In \S~\ref{subsec:CHn} we recall standard material on complex hyperbolic spaces, with emphasis on its solvable model and the root space decomposition of the isometry Lie algebra. In \S~\ref{subsec:II} we derive a formula for the second fundamental form of orbits of subgroups of a parabolic subgroup of the isometry group of $\mathbb{C} H^n$. Finally, in \S~\ref{subsec:reduction} we reduce the classification of minimal homogeneous submanifolds to the study of such orbits.

\subsection{The complex hyperbolic space}\label{subsec:CHn}
Here we describe the algebraic structure of $\mathbb{C} H^n$ as a symmetric space of noncompact type. For a detailed reference, see~\cite{DRDVSL:saopaulo}. 

We denote by $\mathbb{C} H^n$ the {complex hyperbolic space} of complex dimension $n$ (real dimension $2n$) and constant holomorphic sectional curvature $-1$. It is both a Hadamard  and a Kähler manifold; we denote its {complex structure} by $J$.
It is also a Riemannian symmetric space of noncompact type, which can be expressed as the homogeneous space $G/K$, where $G = \mathrm{SU}(1,n)$ and $K=\mathrm{S(U(1)\times U}(n))$. Let $\g{g} = \g{k} \oplus \g{p}$ be the {Cartan decomposition} of the Lie algebra $\g{g}$ of $G$ with respect to a fixed point $o \in \mathbb{C} H^n$, where $\g{k}$ is the Lie algebra of the isotropy group $K$ at $o$, and $\g{p}$ is the orthogonal complement of $\g{k}$ in $\g{g}$ with respect to the Killing form $\mathcal{B}$ of $\g{g}$; note that $\mathcal{B}$ is non-degenerate since $\g{g}$ is simple. Let $\theta$ denote the corresponding Cartan involution. The map $\phi_o\colon G\to \mathbb{C} H^n$, $\phi_o(g)=g\cdot o$, induces an isomorphism $(\phi_o)_{*e}\colon \g{p}\to T_o\mathbb{C} H^n$ under which the Riemannian metric of $T_o\mathbb{C} H^n$ pulls back to a positive multiple $\mathcal{Q}$ of $\mathcal{B}$ restricted to $\g{p}$.  We define a positive definite inner product $\mathcal{Q}_\theta$ on $\g{g}$ by $\mathcal{Q}_\theta( X, Y ) = -\mathcal{Q}(\theta X, Y)$ for any $X, Y \in \g{g}$. 

Let $\g{a}$ be a maximal abelian subspace of $\g{p}$. Since $\mathbb{C} H^n$ is a rank-one symmetric space, $\dim \g{a} = 1$. The endomorphisms in the family $\{ \ad(H): H \in \g{a} \}$ are self-adjoint with respect to $\mathcal{Q}_\theta$ and commute with each other. Consequently, they diagonalize simultaneously, giving rise to the {restricted root space decomposition} $\g{g} = \g{g}_{-2 \alpha} \oplus \g{g}_{-\alpha} \oplus \g{g}_0 \oplus \g{g}_{\alpha} \oplus \g{g}_{2 \alpha}$. Here, the {restricted root spaces} are defined as $\g{g}_\lambda = \{ X \in \g{g} : \ad(H) X = \lambda(H) X, \, \text{for all } H \in \g{a}\}$, for $\lambda \in \g{a}^*$. Note that $\theta\g{g}_\lambda=\g{g}_{-\lambda}$ and $[\g{g}_\lambda,\g{g}_\mu]\subset\g{g}_{\lambda+\mu}$ for any $\lambda,\mu\in\g{a}^*$. Moreover, the root space $\g{g}_{0}$ decomposes as the $\mathcal{Q}_\theta$-orthogonal sum $\g{g}_0= \g{k}_{0} \oplus \g{a}$, where $\g{k}_{0} \cong \g{u}(n-1)$ is the centralizer of $\g{a}$ in $\g{k}$. Then, $\g{k}_0$ normalizes $\g{g}_\alpha$ and centralizes $\g{g}_{2\alpha}$.

The {Iwasawa decomposition theorem} states that $\g{g} = \g{k} \oplus \g{a} \oplus \g{n}$ is a vector space direct sum, where $\g{n} = \g{g}_{\alpha} \oplus \g{g}_{2 \alpha}$. The subspace $\g{n}$ is a $2$-step nilpotent Lie subalgebra of $\g{g}$ with center $\g{g}_{2\alpha}\cong \mathbb{R}$, and $\g{a}$ normalizes  $\g{g}_\alpha$ and $\g{g}_{2\alpha}$. Let $AN$ denote the connected solvable Lie subgroup of $G$ with Lie algebra $\g{a} \oplus \g{n}$. The Iwasawa decomposition at the Lie group level states that $G=KAN$. Hence, the group $AN$ acts simply transitively on $\mathbb{C} H^n$. It turns out that $\mathbb{C} H^n$ is isometric to the Lie group $AN$ equipped with a suitable left-invariant metric $\langle \cdot, \cdot \rangle$.
Under this identification, the complex structure $J$ becomes a left-invariant orthogonal, skew-adjoint endomorphism of $\g{a}\oplus \g{n}$ satisfying $J \g{a} = \g{g}_{2\alpha}$ and leaving $\g{g}_\alpha$ invariant.

Let $B \in \g{a}$ and $Z \in \g{g}_{2\alpha}$ be unit vectors with respect to the metric $\langle \cdot, \cdot \rangle$ such that $JB = Z$. Then $\g{a} = \mathbb{R} B$ and $\g{g}_{2\alpha} = \mathbb{R} Z$.
It is crucial to distinguish between the inner product $\mathcal{Q}_\theta$ on $\g{g}$ derived from the Killing form and the inner product $\langle \cdot, \cdot \rangle$ on $\g{a}\oplus\g{n}$ that determines the above-mentioned left-invariant metric on $AN$. The relation is given by:
\begin{equation}\label{eq::rel_formaKilling}
    \mathcal{Q}_\theta(B,B)=\langle B, B \rangle \quad \text{and} \quad \mathcal{Q}_\theta(X,Y)=2 \langle X, Y \rangle, \quad \text{for any } X, Y \in \g{n}.
\end{equation}
Moreover, we have the following bracket relations for the Lie algebra $\g{a} \oplus \g{n}$:
\begin{equation}\label{eq:brackets}
\begin{aligned}
[B,U]=\frac{1}{2}U, \quad [B,Z]=Z, \quad [Z,U]=0, \quad [U,V]=\langle JU,V \rangle Z,
\end{aligned}
\end{equation}
for any $U, V \in \g{g}_\alpha$. Note that we can identify $\g{g}_\alpha \cong \mathbb{C}^{n-1}$ via the complex structure $J|_{\g{g}_\alpha}$.

Next, we gather some bracket relations involving the Cartan involution $\theta$. 

\settasks{
	column-sep=-2.2em
}

\begin{lemma}\label{lemma_id_brackets_ttheta}
    Let $U, V \in \g{g}_{\alpha}$ and $Z \in \g{g}_{2\alpha}$. The following identities hold:
    \begin{tasks}(3)
        \task $[\theta U,B] = \frac{1}{2}\theta U$.
                        \task $[\theta Z,B] = \theta Z$.
        \task $[\theta U, V] = \langle U,V \rangle B \pmod{\g{k}_0}$.

                     \task $[\theta U, Z] = -JU$.
        \task $[\theta Z,  Z] = 2B$.
    \end{tasks}
\end{lemma}

\begin{proof}
    (a) Using that $\theta$ is an involutive Lie algebra automorphism, $\theta B = -B$ and \eqref{eq:brackets}, we have:
    \[
    [\theta U,B] = \theta [U,\theta B] =  -\theta [U,B] = -\theta \Bigl(-\frac{1}{2}U\Bigr) = \frac{1}{2}\theta U.
    \]

    (b) Analogous to (a).

    (c) By the properties of root spaces, we know that $[\theta U, V] \in \g{g}_{0} = \g{a} \oplus \g{k}_{0}$. 
    Using \eqref{eq::rel_formaKilling}, the relation $\mathcal{Q}_\theta([X,\cdot],\cdot)=-\mathcal{Q}_\theta(\cdot,[\theta X,\cdot])$ for any $X\in\g{g}$ (which follows from the invariance of the rescaled Killing form $\mathcal{Q}$), and \eqref{eq:brackets}, we get:
        \[
    \begin{aligned}
    	\langle [\theta U, V], B \rangle =\mathcal{Q}_\theta([\theta U, V], B) = -\mathcal{Q}_\theta(V,[U,B])= \frac{1}{2} \mathcal{Q}_\theta( U, V) =\langle U, V \rangle.
    \end{aligned}
    \]

    (d) Since $\theta U\in\g{g}_{-\alpha}$ and $Z\in\g{g}_{2\alpha}$, we have $[\theta U,Z]\in\g{g}_\alpha$. Let $V\in \g{g}_\alpha$. Then, using the relations~\eqref{eq::rel_formaKilling}, \eqref{eq:brackets}, and the invariance of $\mathcal{Q}$, we have:
    \[
    2\langle[\theta U,Z],V\rangle= \mathcal{Q}_\theta([\theta U,Z],V)=-\mathcal{Q}_\theta(Z,[U,V])=-\langle JU,V\rangle \mathcal{Q}_\theta(Z,Z)=-2\langle JU,V\rangle.
    \]

    (e) As in (c), $[\theta Z, Z] \in \g{a} \oplus \g{k}_0$. For any $T \in \g{k}_0$, using the invariance of $\mathcal{Q}$, we have
    \[
        \mathcal{Q}_\theta( [\theta Z, Z], T ) =-\mathcal{Q}_\theta(Z, [Z,T]) =0,
    \]
    since $[Z, T] = 0$. Thus, $[\theta Z, Z]$ is proportional to $B$. Calculating the projection yields
    \[
    \begin{aligned}
        \langle [\theta Z, Z], B \rangle &= \mathcal{Q}_\theta([\theta Z, Z], B) = -\mathcal{Q}_\theta(Z, [Z, B])=\mathcal{Q}_\theta(Z,Z)=2\langle Z, Z\rangle=2.
    \end{aligned}\qedhere
    \]
\end{proof}
\subsection{The second fundamental form of a homogeneous submanifold}\label{subsec:II}

We will now derive a convenient formula for the second fundamental form $\mathrm{II}$ of a wide family of homogeneous submanifolds of $\mathbb{C} H^n$, based on a more general formula that holds for every symmetric space of noncompact type. Recall that the mean curvature vector field $\mathcal{H}$ of a submanifold $S$ is given by $\mathcal{H}_p=\sum_{X \in \mathsf{B}_p}\mathrm{II}(X,X)$, where $p\in S$ and $\mathsf{B}_p$ is an orthonormal basis of $T_p S$. Then $S$ is called minimal if $\mathcal{H}$ vanishes identically. To verify whether a homogeneous submanifold is minimal, it suffices to compute the mean curvature vector at one point.

Let $M = G/K$ be a symmetric space of noncompact type with base point $o$ and let $S \subset M$ be a homogeneous submanifold. By homogeneity of $M$, we can assume that $o \in S$. Let $H$ be a connected Lie subgroup of $G$ such that $S = H \cdot o$. We denote by $\g{h} \subset \g{g}$ the Lie algebra of $H$.
Under the canonical identification $T_{o}M \cong \g{p}$, the tangent space $T_{o}S$ corresponds to the subspace $\g{h}_{\g{p}}$ of $\g{p}$, where the subscript denotes projection with respect to the orthogonal sum $\g{k}\oplus\g{p}$ (or equivalently, $\mathcal{Q}_\theta$-orthogonal projection).

The following formula will be essential for our calculations. It was derived by Solonenko for symmetric spaces~\cite[Proposition 2.2.43]{Solonenko_thesis}. See also~\cite[Proposition 2.2]{ADS} for an alternative expression and~\cite[Lemma 2.1]{LN-RV} for a generalization to homogeneous spaces.

\begin{proposition}\label{PropSolonenkoFormula}
Let $M=G/K$ be a symmetric space with base point $o$. Let $H$ be a connected Lie subgroup of $G$, and let $S=H\cdot o$. Then, for any $X,Y \in T_{o}S \cong \g{h}_{\g{p}}$, the second fundamental form of $S$ is given by
\begin{equation*}
\mathrm{II}^{\g{p}}(X,Y)=[X_{\g{k}}, Y]_{\g{h}^{\perp}_{\g{p}}},
\end{equation*}
where $X_{\g{k}} \in \g{k}$ is a vector such that $X+X_{\g{k}} \in \g{h}$, and $(\cdot)_{\g{h}^{\perp}_{\g{p}}}$ denotes the orthogonal projection onto the normal space to $S$ at $o$ regarded as a subspace of $\g{p}$. 
\end{proposition}

As it will follow from \S~\ref{subsec:reduction}, most of our efforts will be devoted to analyze subgroups $H$ of parabolic subgroups of $G=\mathrm{SU}(1,n)$. Since complex hyperbolic spaces have rank one, $G$ has exactly one (proper) parabolic subgroup up to conjugacy. Its identity connected component is of the form $K_0AN$, where $K_0$ is the connected subgroup of $K$ with Lie algebra $\g{k}_0$. In other words, a parabolic subalgebra of $\g{g}$ is of the form $\g{k}_0\oplus\g{a}\oplus\g{n}$. Therefore, it will be convenient to translate Proposition~\ref{PropSolonenkoFormula} into this setting, namely, restricting to subgroups $H$ of $K_0AN$ and taking advantage of the fact that $\mathbb{C} H^n$ is isometric to $AN$ with the left-invariant metric $\langle\cdot,\cdot\rangle$. Thus, we identify the tangent space $T_{o}\mathbb{C}H^n\cong\g{p}$ with the Lie algebra $\g{a}\oplus \g{n}$ through the linear isomorphism $\Psi = \frac{1}{2}(1-\theta)\vert_{\g{a}\oplus \g{n}}$, which is an isometry between $(\g{a}\oplus \g{n}, \langle \cdot, \cdot \rangle)$ and $(\g{p}, \mathcal{Q}_\theta\vert_{\g{p}\times\g{p}})$ (cf.~\cite[Lemma~2.2]{DRDVK:mathz}).

\begin{proposition}\label{prop:Solonenko_an}
    Let $K_0AN$ be the connected component of the identity of a  parabolic subgroup of the isometry group of $\mathbb{C} H^n$. Let $H$ be a connected Lie subgroup of $K_0AN$ with Lie algebra $\g{h}$, and let $S=H\cdot o$. We identify the tangent space $T_{o}S$ with  $\g{h}_{\g{a}\oplus \g{n}}$, the component of $\g{h}$ in $\g{a}\oplus \g{n}$ with respect to $\g{k}\oplus\g{a}\oplus\g{n}$, and the normal space $\nu_{o}S$ with the $\langle \cdot, \cdot \rangle$-orthogonal complement $\g{h}_{\g{a}\oplus \g{n}}^{\perp}$  in~$\g{a}\oplus \g{n}$.
    
    Let  $X = aB+U+xZ$, $Y=bB+V+yZ \in \g{h}_{\g{a}\oplus \g{n}}$ be tangent vectors, where $U,V\in\g{g}_\alpha$ and $a,b,x,y\in\mathbb{R}$. Then, the second fundamental form of $S$ at $o$ is given by
    \begin{equation*}
    \begin{split}
        \mathrm{II}(X,Y) = \left( \Bigl( \frac{\langle U,V \rangle}{2}+xy \Bigr)B -\frac{b}{2}U-\frac{y}{2}JU-\frac{x}{2}JV +[T,V]     -\Bigl(bx+\frac{\langle JV,U\rangle}{2}\Bigr) Z \right)^{\perp},
    \end{split}
    \end{equation*}
    where $T \in \g{k}_0$ is a vector such that $X+T \in \g{h}$, and $(\cdot)^\perp$ denotes $\langle \cdot, \cdot \rangle$-projection onto $\g{h}_{\g{a}\oplus \g{n}}^{\perp}$.
\end{proposition}

\begin{proof}
    First note that $\frac{1}{2}(1\pm\theta)$ are the  orthogonal projections onto $\g{k}$ and $\g{p}$, respectively, with respect to $\mathcal{Q}_\theta$. Then, $\Psi = \frac{1}{2}(1-\theta)\vert_{\g{a}\oplus \g{n}} \colon \g{a}\oplus \g{n} \to \g{p}$ is a linear isometry that maps $\g{h}_{\g{a}\oplus\g{n}}$ onto $\g{h}_\g{p}$. Hence, it intertwines the respective $\g{a}\oplus\g{n}$ and $\g{p}$-projections onto the normal space to $S$ at $o$, that is,  $\Psi(X)_{\g{h}^\perp_{\g{p}}} = \Psi(X^{\perp})$ for any $X \in \g{a}\oplus\g{n}$.
    Thus, the second fundamental form $\mathrm{II}$ on $\g{a}\oplus \g{n}$ is related to the second fundamental form $\mathrm{II}^{\g{p}}$ on $\g{p}$ by $
       \Psi\circ \mathrm{II} =\Psi^*\mathrm{II}^{\g{p}}$.
       
    Let $X,Y\in\g{h}_{\g{a}\oplus\g{n}}$, and let $T \in \g{k}_0$ such that $\tilde{X} = X + T \in \g{h}$. Note that $\tilde{X} =\tilde{X}_{\g{k}} + \Psi(X)$ is the splitting of $\tilde{X}$ with respect to the Cartan decomposition $\g{g}=\g{k}\oplus\g{p}$, where $\tilde{X}_{\g{k}} = \frac{1}{2}(1+\theta)\tilde{X}$ and $\Psi(X)$ are the $\g{k}$ and $\g{p}$-components of $\tilde{X}$, respectively. 
    Applying Proposition~\ref{PropSolonenkoFormula}, we get:
    \begin{equation*}
    \Psi(\mathrm{II}(X,Y))=\mathrm{II}^{\g{p}}(\Psi(X),\Psi(Y))= [\tilde{X}_{\g{k}},\Psi(Y)]_{\g{h}^{\perp}_{\g{p}}}
     = \Bigl[\frac{1}{2}(1+\theta)\tilde{X},\frac{1}{2}(1-\theta)Y \Bigr]_{\g{h}^{\perp}_{\g{p}}}.
    \end{equation*}
	We can expand this expression by using the fact that $\theta$ is an involutive automorphism of~$\g{g}$:
    \begin{equation}\label{eq:PsiII}
    \begin{aligned}
       \Psi\left( \mathrm{II}(X,Y)\right) =\left( \frac{1}{4}(1-\theta) \left( [\tilde{X}, Y] + [\theta \tilde{X}, Y ] \right) \right)_{\g{h}^\perp_{\g{p}}} 
    \end{aligned}
    \end{equation}
    We now compute the term inside the inner parentheses using the structural equations of $\mathbb{C}H^n$ and the identities from Lemma~\ref{lemma_id_brackets_ttheta}. 
    Substituting $\tilde{X} = aB+U+xZ+T$, $\theta \tilde{X}=-aB+\theta U +x\theta Z+T$ and $Y = bB+V+yZ$, and using the fact that $[\g{g}_\lambda,\g{g}_\mu]\subset\g{g}_{\lambda+\mu}$ for any roots $\lambda,\mu$, $[T,B]=[T,Z]=0$ since $T\in\g{k}_0$, the identities in~\eqref{eq:brackets} and Lemma~\ref{lemma_id_brackets_ttheta}, we have
    \[
    \begin{aligned}
         [\tilde{X}, Y] + [\theta \tilde{X}, Y ] &=   b[U,B]+ [U,V] +bx[Z,B]+ 2[T,V]
        +b[\theta U,B]
        \\&\quad+ [\theta U, V] + y[\theta U, Z] 
         +bx[\theta Z,B]+ x[\theta Z, V] + xy[\theta Z, Z] \\
        &= \left(\langle U,V \rangle+2xy \right)B-b\frac{(1-\theta)}{2}U-yJU+x\theta JV+2[T,V]\\
        &\quad-bx(1-\theta)Z-\langle JV,U \rangle Z \pmod{\g{k}_0}.
    \end{aligned}
    \]
	Applying $\frac{1}{4}(1-\theta)$ to this expression and taking into account that $(\frac{1}{2}(1-\theta))^2=\frac{1}{2}(1-\theta)$, the term inside the outer parentheses on the right hand side of~\eqref{eq:PsiII} can be rewritten as
    \begin{equation}\label{eq:proj}
    \begin{aligned}
         \frac{1-\theta}{2}\Bigg( \left( \frac{\langle U,V \rangle}{2}+xy \right)B -\frac{b}{2}U-\frac{y}{2}JU-\frac{x}{2}JV +[T,V]
        -\left(bx+\frac{\langle JV,U\rangle}{2}\right) Z \Bigg).
    \end{aligned}
    \end{equation}
    Note that in this expression, all summands inside the outer parentheses belong to $\g{a}\oplus\g{n}$. Then, we can substitute the projection $\frac{1}{2}(1-\theta)$ by the isometry $\Psi$. Using that $\Psi$ intertwines the  $\g{a}\oplus\g{n}$ and $\g{p}$-projections onto the normal space to $S$ at $o$, from \eqref{eq:PsiII} and~\eqref{eq:proj} we readily get the formula for $\mathrm{II}(X,Y)$ in the statement. 
    \end{proof}

\begin{remark}
Since we are primarily interested in calculating the mean curvature, we will mostly use the formula in Proposition~\ref{prop:Solonenko_an} when $X=Y=aB+U+xZ$, which reduces to
    \begin{equation}\label{eq:SFF_H}
           \mathrm{II}(X,X) = \Bigg( \left( \frac{|U|^2}{2}+x^2 \right)B -\frac{a}{2}U-xJU +[T,U] -axZ \Bigg)^{\perp}.
    \end{equation}
\end{remark}

\subsection{Reduction to subgroups of $K_{0}AN$}\label{subsec:reduction}

We will now show that, in order to obtain a complete classification of minimal homogeneous submanifolds in $\mathbb{C}H^n$,  it suffices to classify the minimal orbits through the base point $o$ of subgroups of the parabolic group $K_{0}AN$. We will actually show this for every rank-one symmetric space of noncompact type.

Let $M=G/K$ be a symmetric space of noncompact type, with $G=\mathrm{Isom}^0(M)$ and isotropy $K$ at a base point $o\in M$. Following~\cite[\S~2.17]{Eberlein}, a parabolic subgroup of $G$ can be geometrically defined as a stabilizer of a point in the ideal boundary $M(\infty)$ of $M$. Note that, with this definition, the total group $G$ is not parabolic. If $M$ has rank one, the isotropy $K$ acts transitively on the ideal boundary $M(\infty)$, so there is only one parabolic subgroup of $G$ up to conjugacy. Its connected component of the identity is precisely the Lie subgroup $K_0AN$ of $G$ with Lie algebra $\g{k}_0\oplus\g{a}\oplus\g{n}$. The definition and properties of the factors $\g{k}_0$, $\g{a}$, $\g{n}$ and the corresponding subgroups for any rank-one symmetric space are analogous to those given in Section~\ref{subsec:CHn} for $\mathbb{C} H^n$ (see~\cite[\S~4.1]{DRDVSL:saopaulo}).

Now, let $S$ be a homogeneous submanifold of a rank-one symmetric space $M=G/K$ of noncompact type. We can express $S$ as an orbit $H \cdot p$, where $H$ is a connected Lie subgroup of the isometry group $G$ and $p \in M$. As proved in \cite[Theorem~6.2]{ADS}, $H$ is contained in a parabolic subgroup of $G$ or $H$ has a totally geodesic orbit. If $H$ has a totally geodesic orbit, it follows from \cite[Proposition~5.5]{ADS} that the minimal orbit $S=H\cdot p$ must be totally geodesic as well. Therefore, we have the following result.

\begin{proposition}\label{prop:reductive}
	Let $S=H\cdot p$ be a minimal homogeneous submanifold of a rank-one symmetric space $M=G/K$. Then, $S$ is a totally geodesic submanifold  of $M$ or $H$ is contained in a  parabolic subgroup of $G$.
\end{proposition}

Thus, from now on we focus on the case where $H$ is contained in a parabolic subgroup of $G$, which up to conjugation can be taken as $K_0AN$.
A priori, one must study orbits of the form $H \cdot p$ for arbitrary $p \in M$ and arbitrary $H \subset K_0 AN$. The following lemma allows us to restrict our attention solely to orbits through the base point $o\in M$.

\begin{lemma}\label{lemma:reduction_K0AN}
    Let $H$ be a subgroup of $K_0 AN$ and $p \in M$. Then, the orbit $H \cdot p$ is congruent to an orbit $\tilde{H} \cdot o$, where $\tilde{H}$ is a subgroup of $K_0 AN$ that is $AN$-conjugate to $H$.
\end{lemma}

\begin{proof}
    Since the group $AN$ acts simply transitively on $M$, there exists a unique element $g \in AN$ such that $g \cdot o = p$. We can rewrite the orbit as:
    \[
        H \cdot p = H \cdot (g \cdot o) = H g \cdot o = g (g^{-1} H g) \cdot o.
    \]
    Let $\tilde{H} = g^{-1} H g$. Then, the orbit $H \cdot p$ is congruent to the orbit $\tilde{H} \cdot o$ via the isometry $g$. Since $H \subset K_0 AN$ and $g \in AN \subset K_0 AN$, it follows trivially that $\tilde{H} = g^{-1}Hg \subset K_0 AN$.
\end{proof}

\section{Proof of the main results}~\label{sec:classification}

In this section we prove the classification result in Theorem~\ref{Thm1} and, as corollaries, Theorems~\ref{th:short} and~\ref{th:isop}. Based on the conclusion of the previous section~\S~\ref{subsec:reduction}, we will consider a connected Lie subgroup $H$ of the isometry group $G$ of $\mathbb{C} H^n$ with Lie algebra $\g{h}$ contained in the parabolic subalgebra $\g{k}_0 \oplus \g{a} \oplus \g{n}$. Our approach relies on the analysis of the possible Lie subalgebras $\g{h}$ in terms of their projections onto the solvable part $\g{a} \oplus \g{n}$ and the compact part $\g{k}_0$, along with the minimality condition.

Let us start by establishing some useful notation  to relate the $\g{a}\oplus\g{n}$ and the $\g{k}_0$-components of vectors in $\g{h}$. Recall that $\g{k}_0\oplus\g{a}\oplus\g{n}$ is a  $\mathcal{Q}_\theta$-orthogonal direct sum.
Given a Lie subalgebra $\g{h}$ of $\g{k}_0\oplus\g{a}\oplus\g{n}$, let us denote by $\pi_{\g{a}\oplus\g{n}}\colon\g{k}_0\oplus\g{a}\oplus\g{n}\to\g{a}\oplus\g{n}$ the  projection map onto $\g{a}\oplus\g{n}$. We decompose $\g{h}$ as a $\mathcal{Q}_\theta$-orthogonal direct sum $    \g{h} = \g{q}\oplus \widetilde{\g{h}}$,
where $\g{q} = \g{h} \cap \g{k}_0 = \ker(\pi_{\g{a}\oplus \g{n}}|_{\g{h}})$, and $\widetilde{\g{h}}$ is the orthogonal complement to $\g{q}$ in $\g{h}$. Then the restriction $\pi_{\g{a}\oplus \g{n}}\vert_{\widetilde{\g{h}}}\colon\widetilde{\g{h}}\to\g{h}_{\g{a}\oplus\g{n}}$ is bijective.  It is important to note that the subspace $\widetilde{\g{h}}$ is not necessarily a Lie subalgebra~of~$\g{g}$. 

In order to relate the $\g{a}\oplus\g{n}$ and $\g{k}_0$-components of $\widetilde{\g{h}}\subset\g{h}$, we define the map 
\[
\Phi \colon \g{h}_{\g{a}\oplus \g{n}} \;\rightarrow\; \widetilde{\g{h}}_{\g{k}_0}, 
\quad
X \;\mapsto\; \Phi(X), \;\text{the unique element of } \widetilde{\g{h}}_{\g{k}_0} \text{ such that } X+\Phi(X)\in \g{h}.
\]
Hereafter, subscripts are used to denote the  components of $\g{h}$ or $\widetilde{\g{h}}$ with respect to $\g{k}\oplus\g{a}\oplus\g{n}$ (or equivalently, the $\mathcal{Q}_\theta$-orthogonal projection onto the subspace indicated in the subscript).
To see that $\Phi$ is well defined, let $T, T'\in \widetilde{\g{h}}_{\g{k}_{0}}$ be such that both $X+T$, $X+T'\in\g{h}$. Their difference $T-T'$ must lie in $\g{h} \cap \g{k}_{0} = \g{q}$, but also in $\widetilde{\g{h}}_{\g{k}_{0}}$. Since $\widetilde{\g{h}}_{\g{k}_{0}} \cap \g{q}=0$, we get  $T=T'$. One can  easily see that $\Phi$ is a linear map. 

Let $\g{m}=\widetilde{\g{h}}\cap(\g{g}_\alpha\oplus\g{k}_0)$. We can consider the $\mathcal{Q}_\theta$-orthogonal sum 
\[
\g{m}=\g{m}_{nd}\oplus\g{m}_d, \quad\text{ where }\g{m}_{nd}=\ker(\pi_{\g{k}_0}\vert_\g{m})\subset\g{g}_\alpha \text{ is the non-diagonal part of }\g{m}.
\]
Projecting onto $\g{g}_\alpha\subset\g{a}\oplus\g{n}$, we have the $\langle\cdot,\cdot\rangle$-orthogonal sum of $\g{s}=\pi_{\g{a}\oplus\g{n}}(\g{m})=\g{g}_\alpha\cap \g{h}_{\g{a}\oplus \g{n}}$:
\[
\g{s}=\g{m}_{nd}\oplus\g{s}_d, \quad\text{ where }\quad\g{s}_{d}=\pi_{\g{a}\oplus\g{n}}(\g{m}_{d}).
\]
Note that $\ker\Phi\vert_{\g{s}}=\g{m}_{nd}$.
Then the restriction $\Phi\vert_{\g{s}_d} \colon \g{s}_d\rightarrow \Phi(\g{s}_d)\subset \widetilde{\g{h}}_{\g{k}_{0}}$
is a vector space isomorphism. Moreover, $\g{m}_d=\{W+\Phi(W):W\in\g{s}_d\}$ and $\g{m}=\{W+\Phi(W):W\in\g{s}\}$.
 
By definition of $\g{q}$ and $\g{m}$, we deduce that $\g{q}\oplus\g{m}$ has codimension $k\in\{0,1,2\}$ in $\g{h}$, where $k=\dim\g{h}_{\g{a}\oplus\g{g}_{2\alpha}}$. Thus, there exist $X_1$, $X_2\in\g{k}_0\oplus\g{a}\oplus\g{n}$ such that $\g{h}=\g{q}\oplus\mathrm{span}\{X_1,X_2\}\oplus\g{m}$. By elementary linear algebra, we can further assume the following: the $\g{g}_{2\alpha}$-component of $X_1$ is zero; the $\g{g}_\alpha$-components of $X_1$ and $X_2$ are orthogonal to $\g{s}$ (and hence to $\g{m}$); the $\g{a}$-components of $X_1$ and $X_2$ are chosen such that the $(\g{a}\oplus\g{n})$-components $(X_1)_{\g{a}\oplus\g{n}}$ and $(X_2)_{\g{a}\oplus\g{n}}$ are orthogonal; and the $\g{k}_0$-components of $X_1$, $X_2$ and $\g{m}$ are $\mathcal{Q}_\theta$-orthogonal to $\g{q}$. Thus, we obtain a convenient decomposition of $\g{h}$, which we state in the following proposition (cf.~\cite[p.~1203]{DRDVK:mathz}). 
Together with the subsequent Remarks~\ref{rem:decomposition} and~\ref{rem:generators}, Proposition~\ref{prop:decomposition} will be used implicitly throughout the arguments below.

\begin{proposition}\label{prop:decomposition}
	Let $\g{h}$ be a Lie subalgebra of $\g{k}_0\oplus\g{a}\oplus\g{n}$. Then we have a direct sum
\[
\g{h}=\g{q}\oplus\mathbb{R} X_1\oplus\mathbb{R} X_2 \oplus \g{m}, \qquad \text{with}\quad X_1=aB+U+T,\quad X_2=bB+V+xZ+R,
\]
where $\g{q}=\g{h}\cap\g{k}_0$; $\g{m}\subset \g{k}_0\oplus\g{g}_\alpha$ is a vector subspace  $\mathcal{Q}_\theta$-orthogonal to $\g{q}$; $a,b,x\in\mathbb{R}$; $U,V\in\g{g}_\alpha$ are orthogonal to $\g{s}=\pi_{\g{a}\oplus\g{n}}(\g{m})$; $T,R\in\widetilde{\g{h}}_{\g{k}_0}\subset\g{k}_0$; and $ab+\langle U,V\rangle=0$. 

\end{proposition}
\begin{remark}\label{rem:decomposition}
Proposition~\ref{prop:decomposition} implies that $\g{q}$ is $\mathcal{Q}_\theta$-orthogonal to  the other direct summands of $\g{h}$, and the $(\g{a}\oplus\g{n})$-components of $X_1$, $X_2$ and $\g{m}$ are mutually orthogonal with respect to  $\langle\cdot,\cdot\rangle$. However, $X_1$, $X_2$ themselves do not need to be $\mathcal{Q}_\theta$-orthogonal to each other or $\mathcal{Q}_\theta$-orthogonal to $\g{m}$ (because of the $\g{k}_0$-component). Note also that $X_1$ and/or $X_2$ may be zero. Actually, we can and will assume that:
\begin{itemize}
	\item either $a=1$ or $X_1=0$ (that is, $a=0$, $U=T=0$), and
	\item either $x=1$ or $X_2=0$ (that is, $b=x=0$, $V=R=0$).
\end{itemize}
\end{remark} 

\begin{remark}\label{rem:generators}
Associated with the previous decomposition, we will consider the generating set for
$\widetilde{\g{h}}$ given by $\left\{X_1,X_2\right\} \cup \{W_{i}+\Phi(W_i)\}_{i=1}^{\dim\g{s}}$, where $\{W_{i}+\Phi(W_i)\}_{i=1}^{\dim\g{s}}$ is  a basis of~$\g{m}$. Moreover, we can and will assume that $\{W_i\}_{i=1}^{\dim\g{s}}$ is an orthonormal basis of $\g{s}$, and hence $\{(X_1)_{\g{a}\oplus\g{n}},(X_2)_{\g{a}\oplus\g{n}}, W_1,\dots, W_{\dim\g{s}}\}$ is an orthogonal generating set for $\g{h}_{\g{a}\oplus\g{n}}$.  We will also suppose that $\{W_{i}\}_{i=1}^{d}$, with $d=\dim\g{s}_d$, is an orthonormal basis of $\g{s}_d$. Hence $\{W_i\}_{i=d+1}^{\dim\g{s}}$ is an orthonormal basis of $\g{m}_{nd}=\ker \Phi\vert_{\g{s}}$.
\end{remark}

Our goal is to show that the case $x=0$ leads to Theorem~\ref{Thm1}(a), whereas the case $x=1$  corresponds to Theorem~\ref{Thm1}(b). To show this, the main difficulty lies in handling the $\g{k}_0$-components of the generators of $\widetilde{\g{h}}$. This is addressed in Lemmas~\ref{lemma:case4_brackets}, \ref{lemasandwich} and, most significantly, in the technically more demanding Lemma~\ref{lemma:structure_formulas}, whose proof hinges on a new and crucial argument. These results, together with the minimality condition, will force all $\g{k}_0$-components of the generators of $\widetilde{\g{h}}$ to vanish, with the sole exception of $T=(X_1)_{\g{k}_0}$.

\begin{lemma}\label{lemma:case4_brackets}
	Assume $a=x=1$. Then, we have:
	\begin{equation*}\label{eq:TVU}
		\langle[T,V], U \rangle = -\frac{b}{2}(1+|U|^2) ( 1 + 2\langle JU, V \rangle),
		\qquad
		\langle[U,R],V\rangle = \frac{1}{2}(b^2+|V|^2) ( 1 + 2\langle JU, V \rangle).
	\end{equation*}
\end{lemma}
\begin{proof}
 Since $\g{h}$ is a Lie subalgebra of $\g{g}$, the following bracket belongs to~$\g{h}$:
	\begin{equation}\label{exp1}
		[B+U+T,bB+V+Z+R] = \frac{1}{2}V-\frac{b}{2}U  + [U,R] + [T,V]+(1 + \langle JU,V \rangle)Z + [T,R].
	\end{equation}
	Note that this vector is orthogonal to $B$. With this in mind, and focusing on the coefficient of $Z$, we can express \eqref{exp1} using the generators of~$\widetilde{\g{h}}$ (which form a basis, since $a=x=1$)~as 
	\begin{equation}\label{exp1_bis}
		-b(1 + \langle JU,V \rangle)(B+U+T)+(1 + \langle JU,V \rangle)(bB+V+Z+R)+\sum_{i=1}^{\dim\g{s}}c_i(W_i+\Phi(W_i))+S,
	\end{equation}
	for certain $c_i\in\mathbb{R}$ and $S\in\g{q}$. Taking inner product with $U$ in~\eqref{exp1} and~\eqref{exp1_bis}, equating both expressions and using $\langle U,V\rangle=-ab=-b$, we obtain
	\[
	-\frac{b}{2}-\frac{b}{2}|U|^2+\langle[T,V],U\rangle=-b(1 + \langle JU,V \rangle)(1+|U|^2),
	\]
	which yields the first formula. Similarly, taking inner product in~\eqref{exp1} and~\eqref{exp1_bis} with~$V$, we~get
	\[
	\frac{1}{2}|V|^2+\frac{b^2}{2}+\langle [U,R],V\rangle = (b^2+|V|^2)(1+\langle JU,V\rangle).\qedhere
	\]
\end{proof}

\begin{lemma}\label{lemma:structure_formulas}
	Assume $a=1$. Then we have:
	\[
	\sum_{i=1}^{d}\langle [\Phi(W_{i}),W_i],U\rangle=\begin{cases}
		-\frac{d}{2} & \text{if } R-bT=0,
		\\
		-\langle JU,V \rangle- \frac{2+d}{2} & \text{if } R-bT\neq 0 \text{ and } R-bT\notin\Phi(\g{s}_d),
		\\
		-\langle JU,V \rangle- \frac{1+d}{2} & \text{if } R-bT\neq 0 \text{ and } R-bT\in\Phi(\g{s}_d).
	\end{cases}
	\]
\end{lemma}
\begin{proof}
	 For each $i=1,\dots,d$, we have
	\begin{equation*}
		[B+U+T,W_i+\Phi(W_i)]=\frac{1}{2}W_i+[U,\Phi(W_i)]+[T,W_i]+\langle JU,W_i \rangle Z+[T,\Phi(W_i)] \in \g{h}.
	\end{equation*}
	If $x=1$, expressing this vector as a linear combination of the generators of $\widetilde{\g{h}}$ yields
	\begin{equation*}
		-b\langle JU,W_i \rangle(B+U+T)+\langle JU,W_i \rangle (bB+V+Z+R) +\sum_{i=1}^{\dim\g{s}}b_{ij}(W_j+\Phi(W_j))+S_i,
	\end{equation*}
	for some $S_i\in\g{q}$, whereas if $x=0$ (and hence $X_2=0$), the expression reduces to the last two addends lying in $\g{m}\oplus\g{q}$. Taking inner product with the unit vector $W_i$ in the previous expressions, we get
	\begin{equation}\label{eq:b_ii}
		b_{ii} = \frac{1}{2} + \langle [U, \Phi(W_i)], W_i \rangle,\qquad \text{for each }i=1,\dots, d.
	\end{equation}
	Similarly, considering the $\g{k}_0$-component of both expressions, we deduce
	\begin{equation}\label{eq:Tphi_expansion}
		[T,\Phi(W_{i})] = \langle JU, W_i \rangle (R-bT) + \sum_{j=1}^{d}b_{ij}\Phi(W_{j}) + S_i,\qquad \text{for each }i=1,\dots, d.
	\end{equation}
	Note that this holds for both $x=0$ (since  this implies $b=0$ and $R=0$) and $x=1$.
	
	Now, let us define the linear endomorphism of $\Phi(\g{s}_d)+ \mathbb{R}(R-bT)\subset\g{k}_0$ given by
	\[
	\psi ( P )= [T,P]_{\Phi(\g{s}_d)+ \mathbb{R}(R-bT)}.
	\]
	Note that $\psi$ is skew-adjoint with respect to the Killing form of $\g{g}$ (since $T \in \g{k}_0$), and thus $\tr\psi=0$.
	Let us proceed by considering cases based on whether $R-bT$ vanishes or is linearly independent of the set $\{\Phi(W_{i})\}_{i=1}^{d}$.
	
	If $R-bT=0$, then $\{\Phi(W_i)\}_{i=1}^d$ is a basis of $\Phi(\g{s}_d)$, and hence from~\eqref{eq:Tphi_expansion} and~\eqref{eq:b_ii} we get
	\[
	0 = \tr\psi = \sum_{i=1}^{d}b_{ii} = \frac{d}{2} + \sum_{i=1}^{d}\langle [U, \Phi(W_i)], W_i \rangle,
	\]
	which proves the first case in the statement. This includes the subcase $x=0$ (and hence $X_2=0$), so we will have $x=1$ and $X_2\neq 0$ from now on.
	
	Next, suppose that $R-bT$ is linearly independent of $\{\Phi(W_{i})\}_{i=1}^{d}$. Hence, $\{\Phi(W_i)\}_{i=1}^d\cup\{R-bT\}$ is a basis of $\Phi(\g{s}_d)+\mathbb{R}(R-bT)$. 
	Since $a=x=1$, projecting equations \eqref{exp1} and~\eqref{exp1_bis} onto $\g{k}_0$ and equating, yields
	\begin{equation*}
		[T, R] = ( 1 + \langle JU, V \rangle)(R-bT) +\sum_{i=1}^{d}c_i\Phi(W_{i}) + S,
	\end{equation*}
	which along with $[T,R-bT]=[T,R]$, \eqref{eq:Tphi_expansion} and~\eqref{eq:b_ii} implies the second case in the statement:
	\[
	0 = \tr\psi =  1 + \langle JU, V \rangle + \sum_{i=1}^{d}b_{ii} = 1 + \langle JU, V \rangle + \frac{d}{2} + \sum_{i=1}^{d} \langle [U, \Phi(W_i)], W_i \rangle.
	\]
	
	Finally, let us consider the case where $R-bT \neq 0$ is linearly dependent of $\{\Phi(W_{i})\}_{i=1}^{d}$. Then we can write $R-bT = \sum_{i=1}^{d}r_{i}\Phi(W_{i})$, so that $W' = \sum_{i=1}^{d}r_{i}W_{i} \in \g{s}_d$ satisfies $W'+R-bT \in\g{m}\subset \g{h}$. Then, the trace condition together with \eqref{eq:Tphi_expansion} and~\eqref{eq:b_ii} yields
	\[
	0 = \tr\psi = \langle JU, W' \rangle + \frac{d}{2} + \sum_{i=1}^{d}\langle [U, \Phi(W_i)], W_i \rangle.
	\]
	To establish the third case in the statement, it remains to show that $\langle JU, W' \rangle = \frac{1}{2} + \langle JU, V \rangle$. Recall that $a=x=1$, and hence $X_1=B+U+T\in\g{h}$ and $X_2=bB+V+Z+R\in\g{h}$. First observe that $X_3=X_2-(W'+R-bT)=bB+ V-W'+ Z+bT\in\g{h}$, and thus
	\begin{equation}\label{eq:Y}
		Y=[X_1,X_3] = \frac{1}{2}(-bU+V-W')  + [T, -bU+V-W']+ (1 + \langle JU, V-W' \rangle) Z  \in \g{h}.
	\end{equation}
	Using the generating set of $\widetilde{\g{h}}$, the coefficients of the basis elements $X_1=B+U+T$ and $X_2=bB+V+Z+R$ are completely determined, and we can express $Y$ in the form
	\begin{equation}\label{eq:Ybis}
		Y=-b(1+\langle JU,V-W'\rangle)X_1+(1+\langle JU,V-W'\rangle)X_2+ \widetilde{W}+\Phi(\widetilde{W})+\tilde{S},
	\end{equation}
	for certain $\widetilde{W}\in\g{s}$ and $\widetilde{S}\in\g{q}$.
	Note that it follows from~\eqref{eq:Y} that  $Y=Y_\g{n}$, and hence $Y \in \ker \Phi$. Then, projecting \eqref{eq:Ybis} onto $\g{n}$ and observing  that $-bX_1+X_2=-bU+V+Z+R-bT$, we deduce $Y=Y_\g{n} \in \ker \Phi \cap (\g{s} \oplus \mathbb{R}(-bU+V+Z))$. 
	Actually, $Y=r(-bU+V+Z)+\widetilde{W}$, with $r=1+\langle JU,V-W'\rangle$.
	This implies that
	\[
	0=\Phi(Y)=r\Phi(-bU+V+Z)+\Phi(\widetilde{W})=r(R-bT)+\Phi(\widetilde{W}),
	\]
	and then we have $\Phi(\widetilde{W})=-r(R-bT)$. Thus, we can write
	\begin{equation}\label{eq:Y_simple}
		Y=r(-bU+V-W'+Z)+rW'+\widetilde{W},
	\end{equation}
	where we have $rW'+\widetilde{W} \in (\ker \Phi)\cap \g{s}$, since $\Phi(rW'+\widetilde{W})=r(R-bT)-r(R-bT)=0$. 
	Projecting the expressions for $Y$ in~\eqref{eq:Y} and~\eqref{eq:Y_simple} onto $\g{g}_{\alpha}$, we obtain:
	\begin{equation}\label{eq:Y_galpha}
		Y_{\g{g}_{\alpha}} =\frac{1}{2}(-bU+V-W') + [T, -bU+V-W']= r(-bU+V-W') + rW'+\widetilde{W}.
	\end{equation}
	Note that $\langle rW'+\widetilde{W},-bU+V-W'\rangle=0$, because $rW'+\widetilde{W}\in(\ker\Phi)\cap\g{s}$, the vectors $U$, $V$  are orthogonal to $\g{s}$, and $W'\in\g{s}_d$ is orthogonal to $\ker\Phi$ by definition. Since we also have $\langle[T, -bU+V-W'] , -bU+V-W'\rangle=0$, from~\eqref{eq:Y_galpha} we deduce that $r=1/2$. Since $r=1+\langle JU,V-W'\rangle$, we conclude that 
	$\langle JU, W' \rangle=\frac{1}{2} + \langle JU, V \rangle$. This completes the proof of the third case in the statement.
\end{proof}

The final ingredient for the proof of our main result, Theorem~\ref{Thm1}, is the following simple but more general lemma.

\begin{lemma}\label{lemasandwich}
	Let $\g{h}'$ be a Lie subalgebra of $\g{a} \oplus \g{n}$, and $\g{q} \subset \g{k}_{0}$ such that $\widehat{\g{h}}=\g{q} \oplus \g{h}'$ is a Lie subalgebra of $\g{k}_{0}\oplus \g{a}\oplus \g{n}$. Let $\g{h}\subset \widehat{\g{h}}$ be a Lie subalgebra of $\g{k}_{0}\oplus \g{a}\oplus \g{n}$ with $\g{h}_{\g{a}\oplus\g{n}}=\g{h}'$. Then $H'\cdot o = H\cdot o$, where $H$ and $H'$ are the connected subgroups of $G$ with Lie algebras $\g{h}$ ~and~$\g{h}'$.
\end{lemma}
\begin{proof}
	Note that $\g{h}$ and $\g{h}'$ are Lie subalgebras of $\widehat{\g{h}}$. Moreover, $\g{h}_{\g{a}\oplus\g{n}}=\g{h}'=\widehat{\g{h}}_{\g{a}\oplus\g{n}}$, so all three Lie algebras have the same $(\g{a}\oplus\g{n})$-projection. Hence, the connected Lie subgroups $H$, $H'$ and $\widehat{H}$ of $K_0AN\subset G$ with respective Lie algebras $\g{h}$, $\g{h}'$ and $\widehat{\g{h}}$ satisfy $T_o(H\cdot o)=T_o(H'\cdot o)=T_o(\widehat{H}\cdot o)$. Therefore, the dimensions of the three orbits agree, and we deduce that $H\cdot o$ and $H'\cdot o$ are open submanifolds of $\widehat{H}\cdot o$. Since $H\cdot o$ and $H'\cdot o$ are complete subsets of $\widehat{H}\cdot o$, they are closed in $\widehat{H}\cdot o$. By connectedness, all three orbits agree.
\end{proof}

\begin{proof}[Proof of Theorem~\ref{Thm1}]
Let $S$ be a minimal homogeneous submanifold of $\mathbb{C} H^n$. By Proposition~\ref{prop:reductive}, $S$ is totally geodesic or $S$ is the orbit of a connected Lie subgroup of a parabolic subgroup of $G$. Note that totally geodesic submanifolds of $\mathbb{C} H^n$ are either totally real $\mathbb{R} H^k$ or complex $\mathbb{C} H^k$, $k\in\{0,\dots, n\}$. A totally geodesic $\mathbb{R} H^k$ of positive dimension (i.e., $k\geq 1$) is precisely the orbit $H\cdot o$ of the connected subgroup $H$ of $G=\mathrm{SU}(1,n)$ with Lie algebra $\g{h}=\g{a}\oplus\g{m}$, where $\g{m}$ is a $(k-1)$-dimensional totally real subspace of $\g{g}_\alpha\cong\mathbb{C}^{n-1}$, which corresponds to Theorem~\ref{Thm1}(a). A totally geodesic $\mathbb{C} H^k$ of positive dimension ($k\geq 1$) coincides with $H\cdot o$, where $H$ has Lie algebra $\g{a}\oplus\g{m}\oplus\g{g}_{2\alpha}$ and $\g{m}\subset\g{g}_\alpha\cong\mathbb{C}^{n-1}$ is a  complex subspace of complex dimension $k-1$, which fits into Theorem~\ref{Thm1}(b). Therefore, by Proposition~\ref{prop:reductive} and Lemma~\ref{lemma:reduction_K0AN}, if $S$ is not totally geodesic, then $S$ is congruent to the orbit through the base point $o\in \mathbb{C} H^n$ of a connected subgroup $H$ of the parabolic subgroup $K_0AN$.

We proceed to compute the mean curvature $\mathcal{H}$ of $H\cdot o$. We assume the notations established above in this section regarding the decomposition of the Lie algebra $\g{h}$ of $H$ (in particular, Proposition~\ref{prop:decomposition} and Remarks~\ref{rem:decomposition} and~\ref{rem:generators}). Recall that the $(\g{a}\oplus\g{n})$-projections of the vectors in the generating set for $\widetilde{\g{h}}$ yield an orthogonal generating set for $\g{h}_{\g{a}\oplus\g{n}}$.
Using~\eqref{eq:SFF_H}, we have:
\begin{align}\label{eq:II_generators}\nonumber
		\mathrm{II}(B+U,B+U) &= \Bigl(\frac{|U|^2}{2}B - \frac{1}{2}U + [T,U]\Bigr)^{\perp},  \hspace{8,6em}\text{if }a=1,
		\\
		\mathrm{II}(W_i,W_i) &= \Bigl(\frac{1}{2}B+[\Phi(W_i),W_i]\Bigr)^\perp,\hspace{6.7em} i=1,\dots, \dim\g{s},
		\\ \nonumber
		\mathrm{II}(bB+V+Z,bB+V+Z) &= \Bigl(\Bigl(\frac{|V|^2}{2}+1\Bigr)B -\frac{b}{2}V - JV + [R,V]-bZ \Bigr)^{\perp}, \text{ if } x=1.
\end{align}

It will be convenient to distinguish three different cases depending on $a$ and $x$.

\smallskip
\textit{Case $a=0$.} 
Hence, $X_1=aB+U+T=0$ and $\g{h}=\g{q}\oplus\mathbb{R} X_2\oplus\g{m}$. Here, either  $X_2=bB+V+xZ+R=0$ (when $x=0$) or $X_2=bB+V+Z+R$ (when $x=1$). In any case, $B-bZ\in\g{h}_{\g{a}\oplus\g{n}}^\perp$ is a nonzero normal vector to $H\cdot o$, so by~\eqref{eq:II_generators} the mean curvature of $H\cdot o$ satisfies
\[
\langle \mathcal{H},B-bZ\rangle =\frac{\dim\g{s}}{2}+ \frac{x}{|bB+V+Z|^2}\left(\frac{|V|^2}{2} +  1+b^2\right).
\]
This is always positive unless $\g{s}=0$ and $x=0$, which implies that $H\cdot o$ cannot be minimal if it has positive dimension.

\smallskip
\textit{Case $a=1$ and $x=0$.} 
Then $\g{h}=\g{q}\oplus\mathbb{R} X_1\oplus\g{m}$. By~\eqref{eq:II_generators}, the mean curvature of $H\cdot o$ is 
\begin{gather}\label{eq:H_case_2}
	\mathcal{H}=\frac{1}{1+|U|^2}\mathrm{II}(B+U,B+U)+\sum_{i=1}^{\dim\g{s}} \mathrm{II}(W_{i},W_{i})\\ \nonumber
	= \left(\left(\frac{|U|^2}{2(1+|U|^2)} + \frac{\dim\g{s}}{2}\right)B -\frac{1}{2(1+|U|^2)}U+\frac{1}{1+|U|^2}[T,U]+\sum_{i=1}^{d}[\Phi(W_{i}),W_{i}]\right)^{\perp},
\end{gather}
where we recall that $d=\dim\g{s}_d$ and $\Phi(W_i)=0$ for $i>d$. 
Now assume $U \neq 0$ for the sake of contradiction. Then $|U|^{2}B - U \in \g{h}_{\g{a}\oplus \g{n}}^{\perp}$ is a nonzero normal vector. Since $\langle [T,U],U\rangle=0$ and using Lemma~\ref{lemma:structure_formulas} for $R-bT=0$ (since $x=0$), we can calculate
\begin{equation}\label{eq:Hxi_case2}
	\langle \mathcal{H}, |U|^{2}B - U  \rangle = 
	|U|^2 \frac{ \dim\g{s}+1}{2}-\sum_{i=1}^{d}\langle [\Phi(W_{i}),W_{i}], U\rangle=|U|^2 \frac{ \dim\g{s}+1}{2}+\frac{d}{2},
\end{equation}
which does not vanish if $U\neq 0$. Hence $U=0$. But then~\eqref{eq:Hxi_case2} or Lemma~\ref{lemma:structure_formulas} imply $d=0$. Therefore, minimality requires $U=0$ and $\g{s}_d=0$.

Then  $\g{h}=\g{q}\oplus\mathbb{R}(B+T)\oplus \g{m}$, with $\g{m}\subset \g{g}_{\alpha}$. Note that $[W,W']=\langle JW,W'\rangle Z\in\g{h}$, for any $W,W'\in\g{m}\subset\g{h}$, which forces $\g{m}$ to be a totally real subspace of $\g{g}_\alpha$, because $Z\notin\g{h}$. Since $[B+T,W]=W/2+[T,W]\in\g{h}$ for any $W\in\g{m}$, we get $[T, \g{m}] \subset \g{h}\cap\g{g}_\alpha=\g{m}$. Note that $[T,\g{q}]=[B+T,\g{q}]\in\g{k}_0\cap\g{h}=\g{q}$, and $[\g{q},\g{m}]\subset\g{m}$.  This ensures that the subspace $\widehat{\g{h}} = \g{q} \oplus \mathbb{R}T\oplus\mathbb{R}B \oplus \g{m}$ is a Lie subalgebra of $\g{k}_0 \oplus \g{a} \oplus \g{n}$.  Lemma~\ref{lemasandwich} allows us to conclude that the orbit $H \cdot o$ coincides with the orbit $H' \cdot o$, where $\g{h}' = \mathbb{R}B \oplus \g{m} \subset \g{a} \oplus \g{n}$.

Finally, as we already mentioned,  the orbit $H\cdot o=H'\cdot o$ turns out to be a totally geodesic $\mathbb{R} H^k$. Indeed, working with $\g{h}'=\g{a}\oplus\g{m}$ instead of $\g{h}$, we get from~\eqref{eq:SFF_H} and Proposition~\ref{prop:Solonenko_an} that $\mathrm{II}(B,B) = \mathrm{II}(W,W) =\mathrm{II}(B,W)= 0$, for all $W \in \g{m}$.  
Consequently, the orbit is totally geodesic. Since $\g{h}'$ is a totally real subspace of $\g{a}\oplus\g{n}$, it is actually a totally geodesic real hyperbolic space $\mathbb{R}H^k$, $k=\dim\g{m}+1$. This corresponds to item (a) in Theorem~\ref{Thm1}.

\smallskip
\textit{Case $a=x=1$.} 
Now, we have $\g{h}=\g{q}\oplus\mathbb{R} X_1\oplus\mathbb{R} X_2\oplus\g{m}$, where $X_1=B+U+T$ and $X_2=bB+V+Z+R$ are linearly independent, and  $\langle U,V\rangle =-b$. Consider the nonzero normal vector $\xi=|U|^2B-U-b(1+|U|^2)Z\in\g{h}_{\g{a}\oplus\g{n}}^\perp$. Then, taking inner products with $\xi$ in \eqref{eq:II_generators}, normalizing appropriately, and adding up, we get that $\mathcal{H}$ satisfies:
\begin{equation}\label{eq:Hxi_case4}
	\begin{split}
		\langle\mathcal{H},\xi\rangle={}& \frac{2+\dim\g{s}}{2}|U|^2-\sum _{i=1}^d \langle [\Phi(W_i),W_i],U \rangle
		\\&
		+\frac{b^2+|U|^2(1+b^2)-2\langle JU,V\rangle-2\langle [R,V],U\rangle }{2(b^2+|V|^2+1)}.
	\end{split}
\end{equation}
Let us distinguish two subcases depending on whether $R-bT$ vanishes or not.

First, assume that $R-bT=0$. Then, by Lemma~\ref{lemma:case4_brackets}, 
\[
0=\langle [R-bT,U],V\rangle= -\langle[U,R],V\rangle+b\langle [T,V],U\rangle=-\frac{1}{2}(2b^2+b^2|U|^2+|V|^2)(1+2\langle JU,V\rangle),
\] 
and thus either $V=0$ and $b=0$, or $\langle JU,V\rangle=-1/2$. If $\langle JU,V\rangle=-1/2$, again Lemma~\ref{lemma:case4_brackets} implies $\langle[R,V],U\rangle=0$. But then the last addend in~\eqref{eq:Hxi_case4} is strictly positive (with independence of the values of $U$, $V$ and $R$). Then, by Lemma~\ref{lemma:structure_formulas}, the whole expression for $\langle \mathcal{H},\xi\rangle$ in~\eqref{eq:Hxi_case4} is strictly positive, which contradicts the minimality of $H\cdot o$. Hence, we have $V=0$ (and thus $b=0$) and, again by~\eqref{eq:Hxi_case4} and Lemma~\ref{lemma:structure_formulas}, we get $U=0$ and $d=0$. Therefore, we have shown that minimality in this subcase implies $U=V=R=0$, $b=0$ and $\g{s}_d=0$.

Now assume that $R-bT\neq 0$. Using Lemmas~\ref{lemma:case4_brackets} and~\ref{lemma:structure_formulas} in~\eqref{eq:Hxi_case4}, and putting $k=1/2$ if $R-bT\in\Phi(\g{s}_d)$ and $k=1$ otherwise, we get
\begin{align*}
	\langle\mathcal{H},\xi\rangle &= 
	\frac{2+\dim\g{s}}{2}|U|^2+\langle JU,V\rangle +\frac{d}{2}+k
	\\&\quad\,+\frac{b^2+|U|^2(1+b^2)-2\langle JU,V\rangle-(b^2+|V|^2)(1+2\langle JU,V\rangle) }{2(b^2+|V|^2+1)}
	\\
	&\geq 
	\frac{2+\dim\g{s}}{2}|U|^2+ \frac{d}{2}+\frac{1}{2}
	+\frac{|U|^2(1+b^2)-|V|^2 }{2(b^2+|V|^2+1)}
	\\
	&=
	\frac{2+\dim\g{s}}{2}|U|^2+ \frac{d}{2}
	+\frac{|U|^2(1+b^2) +b^2+1}{2(b^2+|V|^2+1)}.
\end{align*}
This expression never vanishes, which contradicts the minimality of $H\cdot o$, so the subcase $R-bT\neq 0$ is impossible.

Therefore, we have shown that  $U=V=R=0$, $b=0$ and $\g{s}_d=0$. In other words, $\g{h}=\g{q}\oplus\mathbb{R}(B+T)\oplus \g{m}\oplus \mathbb{R}Z$ with $\g{m}\subset \g{g}_{\alpha}$. Since $[B+T,W]=W/2+[T,W]\in\g{h}$ for each $W\in\g{m}$, we see that $T$ normalizes $\g{m}$. Similarly, $[\g{q},\g{m}]\subset\g{m}$ and $[\g{q},T]\subset\g{q}$. This implies that  $\widehat{\g{h}}=\g{q}\oplus \mathbb{R}T\oplus\mathbb{R}B \oplus \g{m}\oplus \mathbb{R}Z $ is a Lie subalgebra of $\g{k}_{0}\oplus \g{a}\oplus \g{n}$. Thus, by Lemma~\ref{lemasandwich}, the orbit $H\cdot o$ coincides with $H'\cdot o$, where $\g{h}'=\g{a} \oplus \g{m}\oplus \g{g}_{2\alpha}$. 

In this case, it is known, and can easily be proved using  Proposition~\ref{prop:Solonenko_an}, that $H\cdot o=H'\cdot o$  is always a minimal submanifold of $\mathbb{C} H^n$, and it is totally geodesic (actually, a totally geodesic $\mathbb{C} H^{k}$ with $k=\dim\g{m}+1$) if and only if $\g{m}$ is a complex subspace of $\g{g}_\alpha\cong\mathbb{C}^{n-1}$; see~\cite[pp.~1039 and 1041]{DRDV:mathz} or~\cite[pp.~757-758]{DRDVSL:adv}. All in all, we have shown that the assumption $a=x=1$ leads us to Theorem~\ref{Thm1}~(b). This concludes the proof of Theorem~\ref{Thm1}.
\end{proof}

Before turning to the proofs of Theorems~\ref{th:isop} and~\ref{th:short}, which relate minimal homogeneous submanifolds to isoparametric families of hypersurfaces, we briefly recall some known facts about the latter.

\begin{remark}\label{rem:isop}
	Isoparametric families of hypersurfaces of $\mathbb{C} H^n$ were classified in~\cite{DRDVSL:adv}.~Each such family has at most one focal leaf. The only families without a focal leaf are, up to congruence, the horosphere foliation (given by the $N$-action on $\mathbb{C} H^n$) and the so-called solvable foliation, given by the action of $H$ as in Theorem~\ref{Thm1}(b) with $\g{m}$ a hyperplane of $\g{g}_\alpha$. All leaves of a horosphere foliation are mutually congruent, and none of them is minimal, whereas the solvable foliation has exactly one minimal leaf called a Lohnherr hypersurface or a fan~\cite{LR}. The families with a focal submanifold are given by tubes around such submanifold, which can be a totally geodesic $\mathbb{R} H^n$ or $\mathbb{C} H^k$, $k\in\{0,\dots,n-1\}$, or the orbit $H\cdot o$ as in~Theorem~\ref{Thm1}(b) for $\dim\g{m}< 2n-3$. Such focal submanifolds are homogeneous, but the associated isoparametric hypersurfaces are inhomogeneous (with nonconstant principal curvatures) except if the normal space to the focal set has constant K\"ahler angle. The focal sets are always minimal~\cite[Theorem~1.1]{GT}, but none of the tubes around them is minimal, as follows from  the expression for their mean curvature in~\cite[p.~1041]{DRDV:mathz}.
\end{remark}

\begin{proof}[Proof of Theorem~\ref{th:isop}]
	Let $S\subsetneq\mathbb{C} H^n$ be a minimal homogeneous submanifold. If $\dim S=0$, then tubes around the point $S$ define an isoparametric family of hypersurfaces~\cite[Main Theorem~(i)]{DRDVSL:adv}, which fits into Theorem~\ref{th:isop}(3). If $\dim S>0$, by Theorem~\ref{Thm1}, $S$ is congruent to the orbit $H\cdot o$ of the connected subgroup $H$ of $G=\mathrm{SU}(1,n)$ with Lie algebra $\g{h}=\g{a}\oplus\g{m}$ (with $\g{m}$ totally real in $\g{g}_\alpha$) or  $\g{h}=\g{a}\oplus\g{m}\oplus\g{g}_{2\alpha}$ (with $\g{m}$ a proper subspace of $\g{g}_\alpha$). 
	
	In the first case, $H\cdot o$ is a totally geodesic $\mathbb{R} H^k$, with $k=1+\dim\g{m}$, as argued above. This corresponds to item Theorem~\ref{th:isop}(1) if $1\leq k\leq n-1$. If $k=n$, we recover a totally geodesic $\mathbb{R} H^n$, which is known to be a focal submanifold of a (homogeneous) isoparametric family of hypersurfaces~\cite[Main Theorem~(ii)]{DRDVSL:adv}. 
	
	In the second case, it is known that $H\cdot o$ is the only minimal orbit of the $H$-action on $\mathbb{C} H^n$~\cite[p.~1041]{DRDV:mathz}. We can distinguish two cases (recall that $\g{g}_\alpha\cong\mathbb{C}^{n-1}$):
	\begin{enumerate}[(i)]
		\item $\dim\g{m}=2n-3$. Then $H\cdot o$ is a homogeneous minimal Lohnherr hypersurface. Two choices of a hyperplane $\g{m}$ in $\g{g}_\alpha$ yield congruent hypersurfaces. The $H$-action on $\mathbb{C} H^n$ defines a regular isoparametric foliation with exactly one minimal leaf (namely $H\cdot o$). This case corresponds to Theorem~\ref{th:isop}(2).
		\item $\dim\g{m}<2n-3$. Then the tubes around $H\cdot o$ determine an isoparametric family of hypersurfaces of $\mathbb{C} H^n$ with focal submanifold $H\cdot o$, as shown in~\cite{DRDV:mathz}. This corresponds to Theorem~\ref{th:isop}(3).
	\end{enumerate}
	Let us prove the converse. Let $S$ be a homogeneous submanifold of $\mathbb{C} H^n$. Clearly, if $S$ is a totally geodesic $\mathbb{R} H^k$ or a  Lohnherr hypersurface, then it is minimal. Suppose $S$ is the focal submanifold of an isoparametric family of hypersurfaces of $\mathbb{C} H^n$. By a general result of isoparametric families of hypersurfaces~\cite[Theorem~1.1]{GT}, $S$ is minimal. 
	\end{proof}
	
	\begin{proof}[Proof of Theorem~\ref{th:short}]
	Let $S$ be a homogeneous submanifold of $\mathbb{C} H^n$. Assume it is minimal. Then it is either the whole ambient space $\mathbb{C} H^n$, or, by Theorem~\ref{th:isop}, a totally geodesic $\mathbb{R} H^k$, $k\in\{0,\dots, n-1\}$, or a Lohnherr hypersurface (which is a minimal leaf of an isoparametric family of hypersurfaces of $\mathbb{C} H^n$), or a focal submanifold of an isoparametric family of hypersurfaces of $\mathbb{C} H^n$ (which is always minimal~\cite{GT}). This proves the result,  since the other implication of the statement holds trivially.
\end{proof}

\end{document}